\documentclass[11pt, a4paper]{amsart}

\usepackage{a4, a4wide}
\usepackage{amssymb}
\usepackage{amsmath}
\usepackage{amsthm}
\usepackage{amstext}
\usepackage{amscd}
\usepackage{latexsym}
\usepackage{graphics}
\usepackage{color}
\usepackage{enumitem}
\usepackage[all]{xy}
\usepackage{tikz-cd}
\usepackage[textsize=scriptsize]{todonotes}

\usepackage[colorlinks, ]{hyperref} 
\hypersetup{
  colorlinks=true,
  citecolor=blue,
  linkcolor=blue,
  urlcolor=blue}

\makeatletter
\def\@tocline#1#2#3#4#5#6#7{\relax
  \ifnum #1>\c@tocdepth 
  \else
    \par \addpenalty\@secpenalty\addvspace{#2}%
    \begingroup \hyphenpenalty\@M
    \@ifempty{#4}{%
      \@tempdima\csname r@tocindent\number#1\endcsname\relax
    }{%
      \@tempdima#4\relax
    }%
    \parindent\z@ \leftskip#3\relax \advance\leftskip\@tempdima\relax
    \rightskip\@pnumwidth plus4em \parfillskip-\@pnumwidth
    #5\leavevmode\hskip-\@tempdima
      \ifcase #1
      \or\or \hskip 2em \or \hskip 2homologyem \else \hskip 3em \fi%
      #6\nobreak\relax
    \dotfill\hbox to\@pnumwidth{\@tocpagenum{#7}}\par
    \nobreak
    \endgroup
  \fi}
\makeatother

\theoremstyle{plain}

\newtheorem{theorem}{Theorem}[section]

\newtheorem{lemma}[theorem]{Lemma}
\newtheorem{corollary}[theorem]{Corollary}
\newtheorem{proposition}[theorem]{Proposition}

\theoremstyle{definition}

\newtheorem{notation}[theorem]{Notation}
\newtheorem{remark}[theorem]{Remark}

\newtheorem{definition}[theorem]{Definition}
\newtheorem{example}[theorem]{Example}

\newtheorem{claim}[theorem]{Claim}
\numberwithin{equation}{section}

\newcommand{\hocolim}{\mathop{\rm hocolim}\limits} 

\newcommand{\inj}{\hookrightarrow}

\newcommand{\Ln}{{\rm L_{\rm Nis}}}

\newcommand{\Pic}{{\rm Pic}}

\newcommand{\Hom}{{\rm Hom}}

\newcommand{\Spec}{{\rm Spec \,}}
\newcommand{\Sing}{{\rm Sing}}

\renewcommand{\tilde}{\widetilde}

\newcommand{\Z}{{\mathbb Z}}
\newcommand{\A}{{\mathbb A}}
\renewcommand{\P}{{\mathbb P}}

\newcommand{\ssk}{\Delta^{op}(\@S h(Sm/k))}
\newcommand{\sA}{\@A}
\newcommand{\sB}{\@B}
\newcommand{\sC}{\@C}
\newcommand{\sO}{\@O}
\newcommand{\sX}{\@X}
\newcommand{\sY}{\@Y}
\newcommand{\sF}{\@F}
\newcommand{\sZ}{\@Z}
\newcommand{\sH}{\@H}
\newcommand{\sU}{\@U}

\def\<{\langle}
\def\>{\rangle} 
\def\-{\overline} 
\def\~{\widetilde}
\def\^{\widehat}

\def\@{\mathcal}
\def\!{\mathscr}
\def\#{\mathbb}
\def\_{\underline}

\input{xy}
\xyoption{all}

\begin{document}
\title{$\A^1$-connected stacky curves and the Brauer group of Moduli of elliptic curves}



\author{Neeraj Deshmukh}
\address{Institute of Mathematics, Polish Academy of Sciences,Ul.$\acute{S}$niadeckich 8, 00-656 Warsaw, Poland}
\email{ndeshmukh@impan.pl}

\author{Suraj Yadav}
\address{Universit\"{a}t Regensburg, Universit\"{a}tsstr. 31, 93040, Regensburg, Germany}
\email{suraj.yadav@ur.de}

\date{\today}

\begin{abstract}
Given a smooth scheme $X$ with an action by an affine algebraic group $G$, we give a formula to compute the Nisnevich sheaf of the motivic connected components of the quotient stack $[X/G]$ in the case of an orbifold. We apply it to identify all the $\A^1$-connected stacky curves and prove the $\A^1$-connectedness of the moduli stack of elliptic curves. We also prove homotopy purity for the smooth algebraic stacks and recover a computation of the Picard group of stacky orbifold curves by computing their mixed motive.
\end{abstract}

\maketitle

\section{Introduction}

 The program to understand the unstable $\A^1$-homotopy type of varieties was initiated in \cite{amorel} with a particular focus on the case of surfaces. In particular all the $\A^1$-connected surfaces were identified and classified upto their $\A^1$-homotopy type. The case of threefolds is quite difficult and no description of all the $\A^1$-connected threefolds is known, though some partial results about classification of projective bundles on projective plane were proved in \cite{akw1}. On the other hand unstable $\A^1$- homotopy type of stacks in the spirit of \cite{amorel} has not been studied and the aim of this project is to understand the homotopy type of algebraic stacks. As a preliminary step we focus on understanding and computing homotopy sheaves $\pi_{i}^{\A^1}([X/G])$ of a given quotient stack $[X/G]$. These homotopy sheaves are the motivic analogue of the homotopy groups of a topological space in algebraic topology. F. Morel proved in \cite{Morel-book} that for a given space $\sY$ over a field, $\pi_i^{\A^1}(\sY)$ is strongly (resp. strictly) $\A^1$-invariant for $i=1$ (resp. $>1$), (See \cite[Definition 1.7, page 2]{Morel-book} for the definition) and conjectured that $\pi_0^{\A^1}(\sY)$ is an $\A^1$-invariant sheaf. Recently a counterexample was given in \cite{ayoub}. However no counterexamples are known to exist for the schemes or stacks yet. While for curves $\pi_0^{\A^1}$ is easily seen to be $\A^1$-invariant, the case of surfaces is already non trivial (see for instance \cite{bs}). In case of schemes $\pi_0^{\A^1}$ carries useful geometric information. For instance $\A^1$-connected proper smooth schemes are rationally connected. Furthermore in characteristic $0$, $\A^1$-connected smooth schemes have trivial Brauer group and $\acute{e}$tale fundamental group (\cite[Proposition 4.1.1, 4.1.2]{amorel}).

For classifying stacks $B_{\acute{e}t}G$ some computations have been done in \cite[Corollary 3.2, page 137]{mv} and \cite[Theorem 1.3]{ekw} with different conditions on the group. In particular $\pi_0^{\A^1}$ of such stacks is $\A^1$-invariant.  A natural generalisation of these results about classifying stacks is to compute $\pi_0^{\A^1}$ of higher dimensional quotient stacks and check the $\A^1$-invariance of $\pi_0^{\A^1}$.

In \cite{hy} authors study the $\A^1$-connected components of the moduli stack of vector bundles on a curve with an arbitrary fixed determinant and prove that it is $\A^1$-connected. As an application, the authors get a classification of projective bundles on a curve up to $\A^1$-weak equivalence.

In this paper, we begin by studying the quotient orbifolds and write down formulas for computing their $\pi_0^{\A^1}$.  As a first step, we have the following result:

\begin{theorem}\label{main1}
	Let $G$ be an algebraic group satisfying $\A^1$-invariance of torsors over fields and the statement of Grothendieck--Serre conjecture (see Conditions \ref{cond1} and \ref{condti2}). Suppose $G$ acts on a smooth scheme $X$ with trivial generic stabilisers. Let $[X/G]$ denote the quotient stack. Then the induced map $\pi_0^{\A^1}(X)/\pi_0^{\A^1}(G) \to \pi_0^{\A^1}([X/G])$ is an injection of Nisnevich sheaves. Moverover if $BG$ is $\A^1$-connected, $\pi_0^{\A^1}(X)/\pi_0^{\A^1}(G) \to \pi_0^{\A^1}([X/G])$ is an isomorphism.  
\end{theorem}

The construction of the morphism in the above theorem is discussed in the beginning of the Section \ref{sec3}. One pertinent example where $BG$ is $\A^1$-connected is $GL_n$ and for the applications to stacky curves this turns out to be sufficient. 
Next we study gerbes and the simplest examples of these are $\mu_n$-gerbes over $\P^d$ and we obtain the following result

\begin{theorem}\label{main2}
	Let $\P^d(\sqrt[n]{\sO(k)})$, $0 < k <n$ be a non trivial $\mu_n$-gerbe over $\P^d/F$(see Example \ref{2.9} for definition), where $F$ is an algebraically closed field of characteristic 0. Then $$\pi_{0}^{\A^1}(\P^d(\sqrt[n]{\sO(k)})) \simeq \pi_{0}^{\A^1}(B\mu_l)$$ where $l = gcd(k,n) $. In particular $\P^d(\sqrt[n]{\sO(k)})$ is $\A^1$-connected if $k$ and $n$ are coprime.

\end{theorem}

Along the way we also classify $\mathbb{G}_m$-torsors over projective spaces up to their unstable motivic homotopy type (see Proposition \ref{3.6}). One upshot of that computation is that $\mathbb{G}_m$-torsors are hardly ever $\A^1$-connected. Using the Theorem \ref{main1} and the methods used in the proof of Theorem \ref{main2} we obtain the following corollary.
\begin{corollary}
$\overline{\@M}_{1,1}$ is $\A^1$-connected. In particular for the Brauer group of $\overline{\@M}_{1,1}$ over a field $k$ we have\begin{enumerate}
	\item  $\mathrm{Br}(\overline{\@M}_{1,1}) \simeq  \mathrm{Br}(k)$  if char($k$) $=0$
	\item $\mathrm{Br}(\overline{\@M}_{1,1})[\frac{1}{p}] \simeq  \mathrm{Br}(k)$ if  char($k$) $ >3$
\end{enumerate}
\end{corollary}
On a related note Chow-Witt ring of $\overline{\@M}_{1,1}$ has been computed in \cite{DiL}. 

Then we focus our attention on stacky curves and identify all $\A^1$-connected stacky curves in Theorem \ref{4.15}. We now highlight some differences between $\A^1$-connectedness of schemes and stacks. In contrast to the case of proper curves where the only $\A^1$-connected curve is $\P^1$, there are more $\A^1$-connected stacky curves. As one would expect, all $\A^1$-connected stacky curves have $\P^1$ as their coarse moduli space. While $\A^1$-connectedness of proper schemes is a birational invariant \cite[Corollary 2.4.6]{amorel} this is not the case for stacks(\cite[Example 2.9]{hy}). 
In the last two sections we discuss homotopy purity and motivic cohomology for stacks and deduce applications to tame stacky curves. In particular, we compute the Picard group of a stacky orbifold curve (Corollary \ref{5.3}) recovering the result proved in \cite[Theorem 1.1]{lopez}. This will be helpful while considering the motivic homotopy types of the stacks as Picard groups are motivic invariants.

In so far as the unstable $\A^1$-homotopy type is concerned, for curves this coincides with their isomorphism class (as schemes). In particular $\pi_0^{\A^1}$ of a curve is sufficient to determine its unstable $\A^1$-homotopy type. In case of stacky curves it seems reasonable to expect that $\pi_0^{\A^1}$ along with $\pi_1^{\A^1}$ will completely determine unstable $\A^1$-homotopy type. In light of this the computation of $\pi_1^{\A^1}$ is an interesting topic which will be explored in a future work.\\ 

\noindent\textbf{Conventions.} All stacks in the paper are assumed to be separated and tame. All sheaves are assumed to be Nisnevich sheaves and torsors are assumed to be \'{e}tale locally trivial unless stated otherwise. $\pi_0^{\A^1}(BG)$ will always be assumed to be pointed by the trivial torsor.\\

\noindent\textbf{Outline.} The paper is structures as follows: in Section \ref{section-preliminaries}, we review some basics concepts and definitions of motivic homotopy theory (Section \ref{subsection-motivic-basics}) and algebraic stacks (Section \ref{subsection-stacks-basics}) and in Section \ref{sec3}, we prove Theorem \ref{main1} (see Theorem \ref{3.5}).

In Section \ref{section-applications}, we analyse $\mu_n$-gerbes over $\P^n$ (Theorem \ref{theorem-mun-gerbes-Pn}). As a consequence, we show that the Brauer group of $\overline{\@M}_{1,1}$ in characterisitic $0$ is trivial (Corollary \ref{corollary-brauer-M11}). We also get a classification of all $\A^1$-connected stacky curves (Theorem \ref{4.15}).

In the last two sections we discuss homotopy purity and motivic cohomology for stacks and deduce applications to tame stacky curves. In particular, we compute the Picard group of a stacky orbifold curve (Corollary \ref{5.3}) recovering the result proved in \cite[Theorem 1.1]{lopez}.\\

\noindent\textbf{Acknowledgements.} We thank Piotr Achinger, Joseph Ayoub and Marc Hoyois for their comments and suggestions. The second-named author would also like to thank Marc Hoyois for helpful discussions around Lemma \ref{3.7}. We are also grateful to Matthias Wendt  for many suggestions and corrections which improved the paper.\\
The first-named author was supported by the project KAPIBARA funded by the European Research Council (ERC) under the European Union's Horizon 2020 research and innovation programme (grant agreement No 802787). The second-named author acknowledges the support of SFB 1085 Higher Invariants funded by DFG.

\section{Preliminaries}\label{section-preliminaries}
\subsection{Preliminaries on motivic homotopy theory}\label{subsection-motivic-basics}
In this section we give a quick introduction to the definitions and concepts from motivic homotopy theory which we will need in this paper.

Let $S$ be a noetherian scheme of finite Krull dimension and $Sm_S$ be the category of smooth schemes over $S$. Let $Pre(S)$ denote the $\infty$-category of presheaves of (pointed) spaces on $Sm_S$. There are endofunctors 
$$ \Sing, \mathrm{L}_{\mathrm{Nis}} : Pre(S) \to Pre(S)$$ where $\Sing$ is the functor which converts a presheaf of (pointed )space into an $\A^1$-invariant (pointed) presheaf(\cite[{Page 87} ]{mv}and $\mathrm{L}_{\mathrm{Nis}}$ is the Nisnevich sheafification functor. The $\infty$-category of (pointed) motivic spaces, denoted $\sH(S)$ is obtained by localising $Pre(S)$, with respect to $\Sing$ and $ \mathrm{L}_{\mathrm{Nis}} $. In other words we have an adjunction$$ \mathrm{L}_{\mathrm{mot}}: Pre(S)\rightleftarrows  \sH(S) : i$$ where $i$ (the right adjoint here) is the fully faithful inclusion function and $\mathrm{L}_{\mathrm{mot}}:= colim (\mathrm{L}_{\mathrm{Nis}}\circ \Sing)^n$. 
We refer to an element of $\sH(S)$ as a (pointed) motivic space. Localising $Pre(S)$ with respect to $\mathrm{L}_{\mathrm{Nis}}$ one obtains the infinity category of Nisnevich sheaves denoted $\@H^{s}(S)$. One can also work with $\acute{e}$tale topology and use $\acute{e}$tale sheafification functor instead of $\mathrm{ L}_{\mathrm{Nis}} $ to obtain the $\infty$-categories denoted $\sH_{\acute{e}t}(S)$ and $\sH_{\acute{e}t}^s(S)$. We have the following adjunction \cite[Page 130]{mv} 
\begin{equation}\label{equation-MV-adjunction}
i_{\ast}: \sH_{{\acute{e}}t}(S) \rightleftarrows\sH(S): i^{\ast}
\end{equation}
Now we assume $S$ is a field for the rest of this section.
Given any pointed space $\sX \in Pre(S)$,  $\pi_i^{\mathrm{pre}}(\sX)$ is the functor on $Sm_S$ which sends a smooth $U$-scheme to $\pi_i(\sX(U))$(Note $\sX(U)$ is a pointed Kan simplicial set). It's Nisnevich sheafification is denoted as $\pi_i(\sX)$. One defines $\pi_i^{\A^1}(\sX) := \pi_i(\mathrm{ L}_{\mathrm{mot}}(\sX))$. Note that to define $\pi_0$ one need not assume the given space to be pointed. A sheaf of sets $\sF$ on $Sm_S$ is said to be $\A^1$-invariant if for any smooth $S$- scheme $U$, the projection map $p: \A^1_U \to U$ induces and isomorphism  $\sF(U) \to \sF(\A^1_U)$. Given a section $H \in \sF(A^1_U)$ we define $H_0$ and $H_1$ to be $i_0^{\ast} H$ and $i_1^{\ast} H$ respectively, where $i_0^{\ast}$ and $i_1^{\ast}$ are induced from the sections $i_0, i_1: U \to \A^1_U$. Any two given sections $x, y \in \sF(U)$ are said to be directly $\A^1$-homotopic if there exists $H \in \sF(A^1_U)$ such that $H_0 = x$ and $H_1 = y$. We define Nisnevich sheaf $\@S(\@F)$ to be $\pi_0(\Sing\@F(U))$. It follows from the $\Sing$ construction that for any field $L$, two sections $x,y \in \sF(L)$ are $\A^1$-homotopic if and only if their images in $\@S(\sF)(L)$ are equal. There is a notion of ghost homotopy introduced in \cite{bhs} which generalises the notion of $\A^1$-homotopy, defined as follows  
\begin{definition}\cite[Definition 2.7]{bs}
 	Let $\sF$ be a sheaf of sets. Let $U$ be a smooth scheme over $k$. The notion of an $n$-ghost homotopy is defined inductively as follows:
 	\begin{enumerate}
 		\item For $n=0$, a $0$-ghost homotopy is the same as an $\A^1$-homotopy.
 		\item For $n>0$, given $t_1,t_2\in \sF(U)$, an $n$-ghost homotopy connecting $t_0,t_1$ consists of the data
 		\[(V\rightarrow \A^1_U, W\rightarrow V\times_{\A^1_U} V, \tilde{\sigma_0}, \tilde{\sigma_1}, h, h^W ),\]
 		where 
 		\begin{itemize}
 			\item $V\rightarrow \A^1_U$ is a Nisnevich cover of $\A^1_U$.
 			\item $\tilde{\sigma_0},\tilde{\sigma_1}: U\rightarrow V$ are lifts of $\sigma_0,\sigma_1: U\rightarrow \A^1_U$, respectively.
 			\item $W\rightarrow V\times_{\A^1_U} V$ is a Nisnevich cover of $V\times_{\A^1_U} V$.
 			\item $h: V\rightarrow \sF$ where $h\circ \tilde{\sigma_i}=t_i$.
 			\item $h^W$ is an $(n-1)$-ghost chain homotopy connecting the two morphisms $p_0,p_1: W\rightarrow V\times_{\A^1_U} V\rightarrow V\rightarrow \sF$ determined by the two projections $V\times_{\A^1_U} V \rightarrow V$.
 		\end{itemize}
 		\item An $n$-ghost chain homotopy connecting $t_0,t_1$ is a finite sequence of $n$-ghost homotopies starting at $t_0$ and ending at $t_1$.
 	\end{enumerate}
 \end{definition}

Two  elements $x, y \in \sF(L)$ are $n$-ghost homotopic if and only if their images in $\@S^n(\sF)(L)$ are $\A^1$- homotopic \cite[Lemma 2.8]{bs}. There exists a universal $\A^1$-invariant sheaf, denoted $\@L(\sF) :=\lim S^n(\sF)$ \cite[Definition 2.9, Remark 2.15]{bhs} and $\pi_0^{\A^1}(\sF)(L) = \@L(\sF)(L)$. \cite[Theorem 2.2]{brs}. This immediately implies the following proposition which we will appeal to in proof of Theorem \ref{main1}.

\begin{proposition}\label{1.2}
	Two sections $x,y \in \@F(L)$ are equal in $\pi_0^{\A^1}(\sF)(L)$ if and only of they are $n$-ghost homotopic
\end{proposition}
The following lemma follows from the definition of ghost homotopy.
\begin{lemma}
	Let $\sF$ be a sheaf of sets and $U$, $V$ be Henselian local rings. Given $x,y: U \rightrightarrows \sF$ and $w,z: V \rightrightarrows \sF$ such that $x$ and $y$ are $n$-ghost homotopic, and so are $w$ and $z$. Then $ (x,w) : U \sqcup V \rightrightarrows \sF$ and $(y,z) : U \sqcup V \rightrightarrows \sF$ are $n$-ghost homotopic.
\end{lemma}
\begin{lemma}\label{lift1}
	Let $W$ be a Nisnevich cover of $U:= \A^1_L \setminus 0$. Suppose we have $\alpha: W \to X$ and $\beta: (\A^1)^h_0 \to X$ such that $\alpha\vert_F $ and $\beta\vert_F$ are $n$-ghost homotopic, where $F$ is the function field of $(\A^1)^h_0$ . Then images of $ \alpha\vert_1$ and $\beta\vert_0$ are equal in $\pi_0^{\A^1}(X)(L)$.
\end{lemma}
\begin{proof}
	We claim $ \alpha\vert_1$ and $\beta\vert_0$ are $n+1$ ghost homotopic. $ W':= W \sqcup (\A^1)^h_0$ is Nisnevich cover with maps to $X$ from this cover supplied by $\alpha$ and $\beta$ respectively. Write $W = \sqcup_I W_i$ where $W_i$ is an irreducible component. Then $ V=W' \times_{\A^1_L} W' \simeq \sqcup_I \Spec F \sqcup (W \times W) \sqcup (X^h_0 \times X^h_0)$. $pr_1 = (pr_1^F,pr_1^{W}, pr_1^{X^h_0}): V \to W $ is the morphism where $pr_1^F: \sqcup_I\Spec F \to W $, $pr_1^{W}: W\times W \to W$ is the first projection and $ pr_1^{X^h_0}: X^h_0 \times X^h_0 \to X^h_0$ is also the first projection. $pr_2 = (pr_2^F,pr_2^{W}, pr_2^{X^h_0})$ is given by second projections of $W$ and $X^h_0$ and $pr_2^F : \sqcup_I\Spec F \to X^h_0$. By refining $W \times W$, if necessary, composing $pr_i^W$,$ (i=1,2)$ with $W \to U \to X$ give the same section by sheaf condition and similarly for $pr_i^{X^h_0}$. By previous lemma $pr_1$ and $pr_2$ are $n$-ghost homotopic.
\end{proof}
We recall the following lemma which gives rise to an $\A^1$-fiber sequence from \cite{Morel-book} and is an important computational tool for $\pi_i^{\A^1}$
\begin{lemma}\cite[Lemma 6.5, page 208]{Morel-book}
	\begin{enumerate}
		\item A $G$-torsor (in Nisnevich topology) $\sY \to \sX$ with $G$ strongly $\A^1$-invariant sheaf is an $\A^1$-covering.
		\item A $G$-torsor in $\acute{e}$tale  topology with $G$ a finite $\acute{e}$tale $k$-group with order coprime to char(k) is an $\A^1$-covering.
	\end{enumerate}

\end{lemma}
The following definition is standard.
\begin{definition}
	A morphism $\phi: \sF \to \@G $ of pointed sheaves of sets is said to be an injection of pointed sheaves if for any Henselian local ring $O$ and $x, y \in \sF(O)$ such that both $\phi(x)$ and $\phi(y)$ map to the base point of $\@G(O)$ implies $x = y$. In a similar fashion one can define exactness of pointed sheaves of sets.
\end{definition}
\begin{remark} An $\A^1$-covering of pointed space $\sY \to \sX$, with fiber $F$ is a special case of $\A^1$-fibration and gives rise a long exact sequence $$ \cdots \to \pi_i^{\A^1}(\sY,y) \to \pi_i^{\A^1}(\sX,x) \to \pi_{i-1}^{\A^1}(\sF,f) \cdots \pi_0^{\A^1}(F) \to \pi_0^{\A^1}(\sY) \to \pi_0^{\A^1}(\sX)$$ where for $\pi_0^{\A^1}$ exactness means at the level of pointed sheaves of sets.
\end{remark}

\subsection{Preliminaries on Stacks}\label{subsection-stacks-basics}


By an algebraic stack $\sX$ over $S$ we mean an \'{e}tale sheaf of groupoids on $Sch_S$ which admits a smooth representable surjection $X\rightarrow\sX$ from a scheme $X$ (see \cite{alper}). We will always assume that $\sX$ is quasi-separated, i.e, the diagonal of $\sX$ is representable, separated, and quasi-compact. Also, we will mostly deal with Deligne-Mumford stacks (algebraic stacks whose diagonal is unramified). This implies that there exists an \'{e}tale representable surjection $X\rightarrow \sX$ from a scheme $X$.

An algebraic stack $\sX\rightarrow \Spec k$ is said to be \textit{tame} if for any algebraically closed field $L$ and any object $\xi: \Spec L\rightarrow \sX$, the automorphism group scheme $\underline{Aut}_L (\xi)\rightarrow \Spec L$ is linearly reductive (see \cite[Theorem 3.2]{aov}). If $\sX$ has finite inertia (for example, if $\sX$ is Deligne-Mumford), then this is equivalent to the condition that $char(L)$ does not divide the order of $\underline{Aut}_L (\xi)$.

A stack $\sX$ over $S$ is said to be an \textit{fppf gerbe} if 
\begin{enumerate}
	\item for any scheme $T\rightarrow S$, there exists an fppf cover $T'\rightarrow T$ such that $\sX(T')$ is non-empty, and 
	\item given any $x,y\in \sX(T)$, there exists an fppf cover $T''\rightarrow T$ such that $x\simeq y$ in $\sX(T'')$.
\end{enumerate}

For an algebraic stack $\sX$, and let $s\in |\sX|$ be a point in the topological space of $\sX$ \cite[04XE]{stacksproject}. The residual gerbe at $x$ (if it exists) is a reduced, locally Noetherian algebraic stack $\mathcal{G}_x$ together with a monomorphism $\mathcal{G}_x\rightarrow \sX$ such that $|\mathcal{G}_x|$ is a singleton set. A residual gerbe is an fppf gerbe (see \cite[Appendix B]{rydh-devissage}).

Any stack defines a simplicial sheaf on $Sm_S$ by the nerve construction for a groupoid, any stack may be thought of as a simplicial sheaf on $Sch$. In this way, $\sX$ defines an object of $\sH_{\acute{e}t}(S)$. Using the adjunction \ref{equation-MV-adjunction} $i_*(\sX)$ defines an object of $\sH(S)$ (see \cite{hollander}, \cite{motivesdm} for details).

We say that a stack $\sX$ is a quotient stack if $\sX$ is the stack quotient $[X/G]$ of a scheme $X$ by the smooth action of a linear algebraic group.
For any scheme $V$, and morphism $V \to \sX$ is equivalent to a $G$-bundle $P\rightarrow V$ with a $G$-equivariant map $P\rightarrow X$. We have the following obvious lemma which follows immediately from the definitions
\begin{lemma}\label{lift}
	A morphism $\alpha: V \to [X/G]$ (given by a $G$-bundle $P\rightarrow V$ with a $G$-equivariant map $P\rightarrow X$) lifts (not necessarily uniquely) to $X$ if and only if $P$ is a trivial $G$-bundle. 
\end{lemma}

\begin{definition}
	A stacky curve over a field $k$ is a smooth, proper, geometrically connected  Deligne Mumford stack of dimension 1. 
\end{definition}
\begin{remark}
	A stacky curve with a trivial generic stabiliser is uniquely determined by the data of its coarse space and ramification divisor on the coarse space, (see \cite{gers}).
\end{remark}
Now we look at some examples which will be relevant for us.
\begin{example}
	Let $k$ be a field with intergers $n$ and $m$ invertible in $k$. Consider the action of $\mathbb{G}_m$ on $\A^2\setminus\lbrace 0\rbrace$ given by $(x, y)\mapsto (t^nx , t^m y)$ for positive integers $n$,$m$. The quotient stack denoted, $\P(n,m)$ is a Deligne-Mumford stack with $\P^1$ as its coarse moduli space. $\P(1,m)$ has non-trivial stabilisers at a single point. More precisely, the orbit $(0,y)$ in $\A^2\setminus \lbrace 0\rbrace $ has stabiliser $\Z/n\Z$, and every other orbit has trivial stabiliser. In case $n$ and $m$ are corprime and $n,m>1$ the resulting quotient stack has exactly two points with non trivial stablisers- $\Z/m\Z$ and $\Z/n\Z$.
\end{example}
\begin{example}\label{2.9}
	Let $X$ be a scheme (over a field where $n$ is invertible) and $L$ be a line bundle. Let $ B\mathbb{G}_m \to B\mathbb{G}_m$ be induced from the map $\mathbb{G}_m \to \mathbb{G}_m$ which sends $x$ to $x^n$, for some integer $n >1$. We define the root gerbe $X(\sqrt[n]{L})$ to be the following fibered product  $$\xymatrix{
		X(\sqrt[n]{L})\ar[r] \ar[d] & B\mathbb{G}_m \ar[d] \\
		X\ar[r] & B\mathbb{G}_m
	}$$ where $X \to B\mathbb{G}_m$ is induced by the line bundle $L$.
\end{example}
\begin{example}\label{1.7}
	Consider the Kummer exact sequence of $\acute{e}$tale sheaves for $n>1$
	$$ 0 \to \mu_n \to \mathbb{G}_m \to \mathbb{G}_m \to 0$$
	For any scheme $X$ this induces a long exact sequence of $\acute{e}$tale cohomology groups. Using the standard result (over an algebraically closed field) that $H^2_{\acute{e}t}(\P^1,\mathbb{G}_m) =0$, one obtains $H^2_{\acute{e}t}(\P^1,\mu_n) = \Z/n\Z$. Furthermore each element of this group, called $\mu_n$-gerbe over $\P^1$ corresponds to the root gerbe $\P^1(\sqrt[n]{\sO(k)})$, for $0 \leq k < n$. These $\mu_n$-gerbes admit a quotient stack description. The trivial gerbe is simply $\P^1 \times B\mu_n$. For a non trivial gerbe one has $\P^1(\sqrt[n]{\sO(k)}) \simeq Y/\mathbb{G}_m$, $Y$ is the $\mathbb{G}_m$ torsor corresponding to the vector bundle $\sO(k)$. In particular $\P^1(\sqrt[n]{\sO(1)}) \simeq \P^1(n,n)$.
\end{example}

\section{Computation}\label{sec3}
\begin{notation}
	Torsor will always mean an $\acute{e}$tale local torsor unless specified otherwise, throughout the paper. For a scheme $X$ and a point $p \in X$, $X^h_p$ will denote the henselisation of the local ring at $p$. Given a group $G$ with an action on scheme $X$, $[X/G]$ will denote the resulting quotient stack.  
\end{notation}
For rest of the paper we will deal with $G$ an affine algebraic group over the base field $k$ satisfying
\begin{enumerate}
	\item \label{cond1} $H^1_{\acute{e}t}(L, G) \to H^1_{\acute{e}t}(\A^1_L, G)$ is an isomorphism for all field $L/k$.
	\item \label{condti2} For any regular local ring $R$ over $k$, with function field $F$, $H^1_{\acute{e}t}(R, G) \to H^1_{\acute{e}t}(F, G)$ is an injection.
\end{enumerate}
In other words $G$-torsors are $\A^1$-invariant over fields and $G$ satisfies the statement of Grothendieck--Serre conjecture.

Suppose $G$ acts on a smooth scheme $X/k$. 
Denote by $\sX$ the quotient stack $[X/G].$ 
 Then we have the following commutative diagram $$\xymatrix{
	G \times X\ar@<-.5ex>[r] \ar@<.5ex>[r] & X \ar[r] & \sX
}$$ Applying $\pi_0^{\A^1}$ to the above diagram we have an induced action of Nisnevich sheaves $$\pi_0^{\A^1}(G) \times \pi_0^{\A^1}(X) \rightrightarrows \pi_0^{\A^1}(X) $$ and we denote by $ \pi_{0}^{\A^1}(X)/\pi_{0}^{\A^1}(G) $ the cokernel in the category of Nisnevich sheaves of sets. Note since $G$ is a sheaf of groups $\pi_0^{\A^1}(G)$ is also a sheaf of groups. 

\begin{theorem}\label{3.5}
	In case $BG$ is $\A^1$-connected the induced map $ \pi_{0}^{\A^1}(X)/\pi_{0}^{\A^1}(G)  \to \pi_{0}^{\A^1}(\sX)$ is an isomorphism of sheaves.
\end{theorem}
Before we prove the theorem we need a few technical lemmas. The idea of the proof is to first prove the injection on the field valued points. This holds more generally without $BG$ being $\A^1$-connected. To that end we show that  an equality of elements $x, y \in \pi_{0}^{\A^1}(\sX)(L)$ gives rise to a ghost homotopy on the scheme $X$ in case one of the points lift to $X(L)$. The proof boils down to lifting the ghost homotopy between $x$ and $y$ in the stack $\sX$ to a ghost homotopy on $X$.

\begin{lemma}\label{3.3}
	Let $x_1 \neq x_2 \in X(L)$ be two lifts of a given map $x: \Spec L \to \pi_0(\sX)$ Then $x_1 = x_2 \in \pi_0^{\A^1}(X)(L)/\pi_0^{\A^1}(G)(L)$.
\end{lemma}
\begin{proof}
	If $x_1 = x_2 \in \pi_0^{\A^1}(X)(L)$ then we are done so assume to the contrary for the rest of this proof. Since $x$ lifts to $X(L)$, by Lemma \ref{lift} $x$ is given by a trivial $G$-torsor with an $G$- equivariant map $\phi: G \times L \to X $. Therefore any lift to $X(L)$ given by a section $L \to G \times L$ composed with $\phi$. Therefore there exists $g \in G(L)$ such that $g.x_1 = x_2$. We have two cases now.\\ 
	\underline{Case 1}: $g \neq e \in \pi_0^{\A^1}G(L)$, (where $e$ is the identity element of $G(L)$). In that case, by abuse of notation denote again by $g$ the image of $g$ in $\pi_0^{\A^1}(G)(L)$ and we $[g].x_1 = x_2 $. Therefore  $x_1 = x_2 \in \pi_0^{\A^1}(X)(L)/\pi_0^{\A^1}(G)(L)$.\\ \underline{Case 2}: $g = e \in \pi_0^{\A^1}(G)(L)$. In this case we claim $x_1 = x_2 \in \pi_0^{\A^1}(X)(L)$. To see this, $g = e \in \pi_0^{\A^1}(G)(L)$ implies $g$ and $e$ are $n$-ghost homotopic for some $n \geq 0$. Composing the data of this $n$-ghost homotopy with the morphism $ G \xrightarrow{f} G \times X \to X$ (such that $f(k) = (k,x_1)$ and $G \times X \to X$ is the action map) gives a $n$-ghost homotopy between $x_1$ and $g.x_1 =x_2$.
\end{proof}



\begin{lemma}\label{lifttox}
	Given any finitely generated field extension $L/k$ and $x\in X(L)$ and $h \in \pi_0^{\mathrm{pre}}( (\mathrm{L}_{\mathrm{Nis}}\Sing)^n\sX(\A^1_L )) $ such that $h_0$ lifts to $x$. Then $h_1$ lifts to $ y \in X(L)$ and $x=y \in \pi_0^{\A^1}(X)(L)/\pi_0^{\A^1}(G)(L)$
\end{lemma}
\begin{proof}
We proceed by induction on $n$ with base case $n =0$, which implies the existence of a section $h: \A^1_L \to \pi_0(\sX)$ such that $h_0 \in \pi_0(\sX)(L)$ lifts to $x$. By lemma \ref{lift} $h_0$ is a trivial torsor and Condition \ref{cond1} implies $h$ is a trivial torsor so it has a lift $\tilde{h}$ to $X(\A^1_L)$ such that $\tilde{h}_0 =x$. Now $\tilde{h}_1$ is a lift of $y$ and $\A^1$-homotopy between $x$ and $y$ is supplied by $\tilde{h}$, so $x=y \in \pi_0^{\A^1}(X)(L)$. For a general $n$ we proceed as follows:\\
Precomposing the map $ h: \A^1_L \to (\Ln \circ\Sing)^n\sX$ with the map from generic point $L(t) \to \A^1_L$ we obtain an element $a$ in $\pi_0((\Ln\circ\Sing)^n\sX(L(t))) = \pi_0(\Sing \circ (\Ln\Sing)^{n-1}  \sX(L(t))$. By standard spreading out techniques $a$ can be extended to section $\pi_0(\Sing \circ (\Ln\Sing)^{n-1} \sX(U))$ of an open subscheme $U \inj \A^1_L$. With out loss of generality assume $ \A^1_L \setminus U = \{p\}$. Consider the following Nisnevich distinguished square $$\xymatrix{
	\Spec F \ar[r] \ar[d] & \Spec (\A^1_p)^h \ar[d] \\
	U\ar[r] & \A^1_L
}$$ where $F$ is the fraction field of $ \Spec (\A^1_p)^h$.

We have elements $a \in \pi_0(\Sing \circ (\Ln\Sing)^{n-1} \sX(U))$ and $b := h\vert_{\Spec (\A^1_p)^h } \in \pi_0(\Sing \circ (\Ln\Sing)^{n-1} \sX(\A^1_p)^h )) = \pi_0((\Ln\Sing)^n\sX(\A^1_p)^h ))$ such that their restriction to $\pi_0(\Sing \circ (\Ln\Sing)^{n-1} \sX(F))$ (denoted $a_F$ and $b_F$ respectively) is equal. That means there is an $\A^1$-homotopy (that is a morphism $\A^1_F \to (\Ln\Sing)^{n-1} \sX ) $ between $a_F$ and $b_F$. We want to claim that either $a_F$ or $b_F$ lift to $X(F)$ to be able to use the induction argument. We will have two cases depending on whether $p =0$ or not.\\  
\underline{Case 1}: $p \in \A^1_L$ is  $0$. Consider a lift of $b \in \pi_0(\Sing \circ (\Ln\circ\Sing)^{n-1} \sX(\A^1_p)^h )) $ to an element $c \in \pi_0\sX((\A^1_p)^h )$. Note that $c$ restricted to the residue field of $X^h_p$ (which is $L$), denoted $c\vert_L$ and $x$ map to the same element-which is $b\vert_L$ in $\pi_0(\Sing \circ (\Ln\circ\Sing)^{n-1} \sX(L)$. Therefore we have an $\A^1$-homotopy in $(\Ln\circ\Sing)^{n-1}\sX$ and use induction hypothesis to conclude that $c\vert_L$ lifts to $X(L)$ and moreover is equal to $x$ in $\pi_0^{\A^1}X(L)$.  Furthermore by invariance of torsors between Henselian local rings and their residue fields \cite[Proposition 6.1.1(b)]{ces}, $c$ also lifts to $X(X^h_p)$. Hence $b$ lifts to $\pi_0X$ and in particular as does $b\vert_F$. Hence  we can apply induction hypothesis to deduce that $a_F$ and $b_F$ lift to $X(F)$ and their images in $\pi_0^{\A^1}X(F)$ are equal. Therefore by Prop \ref{1.2} they are $k$-ghost homotopic, for some $k \geq 0$. It remains to lift $a$ to $\pi_0(X)$ now, since we have already shown $b$ can be lifted. The map $\Spec F \to U$ factors as $\Spec F \to \Spec X^h_x \to U$, for any $x \in U$. As $a_F$ lifts to $X(F)$,  by Condition \ref{condti2} $a\vert_{ X^h_x} $ lifts to $X$. This implies we have lifted $a$ to $X$ Nisnevich locally on $U$. By Lemma \ref{lift1} we have the required $k+1$ ghost homotopy on $X$.\\ 
\underline{Case 2}: $p \in \A^1_L$ is not $0$. The map $\Spec F \to U$ factors as $\Spec F \to \Spec X^h_0 \to U$. Therefore arguing the same way as in previous case and using by invariance of torsors between Henselian local rings and their residue fields, $a_F$ lifts to $X(F)$. Hence again by induction hypothesis we deduce that $a_F$ and $b_F$ lift to $X(F)$ and their images in $\pi_0^{\A^1}X(F)$ are equal. (All the lifts to $X$here only differ by $G$ action). Hence they are $k$-ghost homotopic, for some $k \geq 0$. Condition \ref{condti2} implies $b$ lifts to $\pi_0(X)$ as $b\vert_F$ does. Also as $a\vert_F$ lifts we can argue as in the previous case to conclude $a$ also lifts Nisnevich locally. This finishes the proof.
\end{proof}

We now prove Theorem \ref{3.5}.

\begin{proof}[Proof of Theorem \ref{3.5}]
 The injectivity of the map on field valued points follows from Lemma \ref{lifttox}. 
	 For an arbitary henselian local ring $Z$ with function field $K$, we have an injection of pointed sets $\pi_0^{\A^1}(\sY)(Z) \to \pi_0^{\A^1}(\sY)(K) $ \cite[Corollary 4.17]{c}.  Moreover in case $\@Y$ is a sheaf of groups $\pi_0^{\A^1}(\@Y)$ is $\A^1$-invariant \cite[Corollary 5.2]{c} and $\pi_0^{\A^1}(\sY)(Z) \to \pi_0^{\A^1}(\sY)(K) $ is a bijection. Combining the above two facts we have an injection for the sheaf of sets $(Z) \to \sF(K)$ where $ \sF:= \pi_{0}^{\A^1}(X)/\pi_{0}^{\A^1}(G)$. Then in the following commutative diagram the left vertical map is an injection because of the discussion of previous two lines, the bottom horizontal map is an injection because of the previous paragrah. Hence the top horizontal arrow in an injection.
	$$\xymatrix{
		\sF(Z) \ar[r] \ar[d] &  \pi_{0}^{\A^1}(\sX)(Z) \ar[d] \\
		\sF(K)\ar[r] &  \pi_{0}^{\A^1}(\sX)(K)
	}$$ 
	For surjection we look at the morphism $\sF \to \pi_0(\sX) \to \pi_0^{\A^1}(\sX)$. The morphism $\pi_0(\sX) \to \pi_0^{\A^1}(\sX)$ is surjective by \cite[Corollary 3.22]{mv}. Given any henselian local ring $R$ and a morphism $\Spec R \to \pi_0(\sX)$, we have a $G$-torsor over $R$ with an $G$-equivariant map to $X$. The torsor lifts to $X$ at the closed point of $R$ iff only the morphism $R \to \sX$ lifts to $X$ \cite[Proposition 6.1.1(b)]{ces}. The torsor lifts to $X$ because $BG$ is $\A^1$ connected and one uses \cite[Theorem 1.3]{ekw}; so  we have a lift of the point to $\pi_0^{\A^1}(X)(R)/\pi_0^{\A^1}(G)(R)$.
\end{proof}

\begin{lemma}\label{new}
	Let $[X/G]$ be a quotient stack a field $k$. Suppose there exists a $k$-point $x\in X$ such that $Stab(x)=G$. Then the canonical map $p:[X/G]\rightarrow BG$ has a section. In particular the induced map $\pi_0^{\A^1}([X/G]) \to \pi_0^{\A^1}(BG)$ is surjective.
\end{lemma}

\begin{proof}
	As $X$ has a $k$-point that is fixed by $G$, the morphisms $\Spec k \overset{x}{\rightarrow}X \rightarrow \Spec k$ are $G$-equivariant and the composition is identity. Taking $G$-quotient, we get $BG\overset{\bar{x}}{\rightarrow} [X/G]\rightarrow BG$. Moreover, the composition of these maps is identity on $BG$. Thus, $\bar{x}: BG \rightarrow [X/G]$ is the desired section.
\end{proof}
\begin{remark}
	The results of this section can be summarised as follows: $X \to [X/G] \to BG$ induces morphisms of sheaves of sets $\pi_0^{\A^1}(X)/\pi_0^{\A^1}(G) \to \pi_0^{\A^1}([X/G]) \to \pi_0^{\A^1}(BG)$ which is always exact in the middle (in the sense of pointed sheaves of sets). Theorem \ref{main1} says the first map $\pi_0^{\A^1}(X)/\pi_0^{\A^1}(G) \to \pi_0^{\A^1}([X/G]) $ is an injection of sheaves of sets (which is a stronger statement than saying injection of pointed sheaves of sets). Lemma \ref{new} gives a sufficient condition for $\pi_0^{\A^1}([X/G]) \to \pi_0^{\A^1}(BG)$ to be surjective.
\end{remark}

\section{Applications}\label{section-applications}
In this section we apply the formula obtained for the $\pi_0^{\A^1}$ of the  quotient stack to compute $\pi_0^{\A^1}$ of stacky curves as well as the stack $\overline{\@M}_{1,1}$.

 \begin{lemma}
Let $\sX$ be an $\A^1$-connected stacky curve. Then the coarse moduli space of $\sX$, denoted $X$, is also $\A^1$-connected.
\end{lemma}
\begin{proof}
	We prove the contrapositive. Assume $X$ is not $\A^1$-connected. Then for some algebraically closed field $L$ we have $x \neq y \in \pi_0^{\A^1}(X)(L)$. Then by the property of coarse moduli spaces $\pi_0(\sX)(L) \simeq X(L)$ so we lift $x$ and $y$ to $\pi_0(\sX)$. Any ghost homotopy between the lifts of $x$ and $y$ can be composed with the map $\sX \to X$ to give a ghost homotopy and hence an equality of $x$ and $y$ in $\pi_0^{\A^1}(X)(L)$ which is a contradiction.
\end{proof}
 
\begin{lemma}\label{onestackypoint}
$\P(m,n)$ with $m,n$ coprime integers is $\A^1$ -connected. 
\end{lemma}
\begin{proof}
	This is immediate from the fact that $\sX \simeq (\A^2 \setminus 0 )/\mathbb{G}_m$, where $\mathbb{G}_m$ acts by $\lambda.(x,y) = (\lambda^mx, \lambda^n y)$ and Theorem \ref{main1}.
\end{proof}

\begin{example}
	In contrast to the above lemma, $\A^1$ with one stacky point, say $B_{\acute{e}t}\mu_n$ is not $\A^1$-connected. Using Lemma \ref{new} one sees that it's $\pi_0^{\A^1}$ is isomorphic to $\pi_0^{\A^1}(B_{\acute{e}t}\mu_n)$. 
\end{example}
\begin{example}
$\P^1$ with two stacky points, with residual gerbe $B\mu_n$ at both the points. Then $\sX = \P^1/\mu_n$. So by Lemma \ref{new}, $\sX$ is not $\A^1$-connected.
\end{example}
The next example illustrates that the data of the coarse moduli along with the ramification divisor is not enough to compute $\pi_0^{\A^1}$ of the stack.
\begin{example}
	Let $C$ be any hyperelliptic curve of genus $g$. Then $C/\mu_2$-where action of $\mu_2$ is induced from the degree $2$ finite map to $\P^1$- is a stacky $\P^1$ with $2g+2$ stacky points. The residual gerbe at each of these stacky points is $B\mu_2$. Since $C$ is $\A^1$-rigid, that is $\pi_0^{\A^1}(C) \simeq C$, Theorem \ref{3.5} implies there is no non constant $\A^1$ inside the stack.
\end{example}
Now we compute homotopy sheaves of  $\mu_n$-gerbes over the projective spaces. We do this by computing $\pi_0^{\A^1}$ and $\pi_1^{\A^1}$. We use the quotient stack  description given in Example \ref{1.7} and apply Theorem \ref{3.5} to compute the $\pi_0^{\A^1}$. To compute $\pi_1^{\A^1}$ we first fix the base point, (one which lifts to the smooth cover-which is a $\mathbb{G}_m$-torsor over $\P^1$, as described in  Example \ref{1.7}  ) and then use the $\A^1$-fiber sequence. As an intermediate step we look at $\mathbb{G}_m$-torsors over $\P^n$ (Prop.\ref{3.6}). 
\begin{lemma}\label{3.7}
	$B_{\acute{e}t}\mu_n \to B\mathbb{G}_m \to B\mathbb{G}_m$ is an $\A^1$-fibration, (where $B\mathbb{G}_m \to B\mathbb{G}_m$ is induced by $n^{th}$ power map $\mathbb{G}_m \to \mathbb{G}_m$). In particular there is an exact sequence of Nisnevich sheaves $$ 0 \to \mu_n \to \mathbb{G}_m \to \mathbb{G}_m \to \pi_0^{\A^1}(B\mu_n) \to 0$$ Moreover any $\mu_n$-root gerbe $Y \to X$ fits in a fiber sequence $$B_{\acute{e}t}\mu_n  \to Y \to X$$
\end{lemma}
\begin{proof}
	Since $B\mathbb{G}_m$ is a motivic space, one has a fiber sequence  $F \to B\mathbb{G}_m \to B\mathbb{G}_m$. It remains to identify $F$ with $B_{\acute{e}t}(\mu_n)$. By Kummer exact sequence we have $B_{\acute{e}t}\mu_n \to BG_m \to BG_m$ and therfore there is an induced map $B_{\acute{e}t}\mu_n \to F$. The fiber sequence induces the long exact sequence of homotopy sheaves $$ 1 \to \pi_1^{\A^1}(F) \to \mathbb{G}_m \to \mathbb{G}_m \to \pi_0(F) \to 1$$ So we have $\pi_i^{\A^1}(F) \simeq \pi_i^{\A^1}(B_{\acute{e}t}(\mu_n))$ induced by the morphism $ B_{\acute{e}t}(\mu_n) \to F$. The second part follows from the fact that a $\mu_n$ root gerbe over $X$ (where $X$ is a scheme or a stack) is a base change of the map $B\mathbb{G}_m \to B\mathbb{G}_m$.
\end{proof}

\begin{definition}\label{redstr}
	Let $H \to G$ be a morphism of affine group schemes over a base scheme $S$. Let $P \to X$ be a principal $H$-bundle, then the principal $G$-bundle induced by $P$ via $H \to G$ is $(G \times P)/H \to X$, where $H$ acts on $G \times P$ by $h.(g,p) = (gh^{-1},hp)$ and the $G$-action is given by $g'(g,p) = (g'g,p)$
\end{definition}

First, for the torsor corresponding to $\sO(1)$, $\A^{n+1} \setminus0 \to \P^n$, there is a short exact sequence $$ 1\to \pi_1^{\A^1}(\A^{n+1} \setminus 0) \to \pi_1^{\A^1}(\P^n) \xrightarrow{\theta} \mathbb{G}_m \to 1$$ When $n \geq 2$, $\theta$ is an isomorphism.
\begin{notation}\label{3.9}
	For the rest of the paper $\theta$ will denote the map  
	$\pi_1^{\A^1}(\P^n) \to \mathbb{G}_m $ described above, $\psi^d$ will denote the $d$-fold map $\mathbb{G}_m \to  \mathbb{G}_m$, $\mathbb{G}_m^d:= \mathrm{Im}(\psi^d)$ and $\rho_{n}: \mathbb{G}_m \to \pi_0^{\A^1}(B\mu_n)$ from Corollary \ref{3.7}.
\end{notation}
The first part of the following proposition is standard, see \cite[4.1.2]{DiL} and the rest of the proposition easily follows from the standard arguments. 
\begin{proposition}\label{3.6}
	Let $X$ be the $\mathbb{G}_m$-torsor over $\P^n$ corresponding to the line bundle $\sO(d)$, $n \neq 0$. Then \begin{enumerate}
		\item $X \simeq (\A^{n+1} \setminus 0)/\mu_d$
		\item $\pi_0^{\A^1}(X) \simeq \pi_0^{\A^1}(B\mu_d)$.
		\item there is an exact sequence of Nisnevich sheaves $$ 1 \to \pi_1^{\A^1}(X) \to \pi_1^{\A^1}(\P^n) \xrightarrow{\phi} \mathbb{G}_m \to \pi_0^{\A^1}(X) \to 1$$ where $\phi$ is the composition of $\theta$ and $ \psi^d: \mathbb{G}_m \to \mathbb{G}_m$. Therefore $\pi_0^{\A^1}(X) \simeq \pi_0^{\A^1}(B\mu_d)$ and for $n \geq 2$, $\pi_1^{\A^1}(X) \simeq \mu_d$
		\item $\pi_i^{\A^1}(X) \simeq \pi_i^{\A^1}(\P^n)$, for $i \geq 2$.
	\end{enumerate}   
\end{proposition}
\begin{proof}
	 Since $X$ is a $\mathbb{G}_m$-torsor over $\P^n$ and as $\P^n$ is $\A^1$-connected, we have an exact sequence $$ 1 \to \pi_1^{\A^1}(X) \to \pi_1^{\A^1}(\P^n) \xrightarrow{\phi} \mathbb{G}_m \to \pi_0^{\A^1}(X) \to 1$$ Therefore $ \pi_0^{\A^1}(X)  \simeq \mathbb{G}_m/\mathrm{Im}(\phi)$. We now compute $\mathrm{Im}(\phi)$. For any scheme $X$, we have $\Pic(X) \simeq \Hom_{\sH(k)}(X,B\mathbb{G}_m) \simeq \Hom_{Gr}(\pi_1^{\A^1}(X), \mathbb{G}_m)$, where $Gr$ is the category of strongly $\A^1$-invariants Nisnevich sheaves. So in the above short exact sequence $ \pi_1^{\A^1}(\P^1) \to \mathbb{G}_m$ is induced via the morphism $X \xrightarrow{\sO(1)} B\mathbb{G}_m \to B\mathbb{G}_m $  where $B\mathbb{G}_m \to B\mathbb{G}_m$ is induced by $\psi^d$. Hence $\phi$ is the composition $ \pi_1^{\A^1}(\P^n) \xrightarrow{\theta} \mathbb{G}_m \to \mathbb{G}_m$. Note $\theta$ is already surjective therefore $\mathrm{Im}(\phi)$ is equal to image of $\mathbb{G}_m$ inside $\mathbb{G}_m$ via $\psi^d$ (denoted $\mathbb{G}_m^d)$, which by Corollary \ref{3.7} is precisely $\pi_0^{\A^1}(B\mu_d)$. As $\pi_1^{\A^1}(\P^n) \simeq \mathbb{G}_m$ for $n \geq 2$, $\pi_1^{\A^1}(X) \simeq \mu_d$. \\
\end{proof}
\begin{lemma}\label{notsurj}
	Let $X$ be an $\A^1$-connected variety or a stack and $Y$ be a non trivial $\mathbb{G}_m$ -torsor over $X$. If for the line bundle $\@L \in \Pic(X)$ corresponding to $Y$ admits an $n$-root for some $n \in \Z$, then $Y$ is not $\A^1$-connected.
\end{lemma}

\begin{proof}
	Consider the $\A^1$-fibre sequence
	\[\mathbb{G}_m\rightarrow Y\rightarrow X\]
	As $X$ is $\mathbb{A}^1$-connected, the associated long exact sequence has the form
	\[\ldots \rightarrow\pi_1^{\A^1} (X)\overset{\@L}{\rightarrow} \mathbb{G}_m\rightarrow \pi_0^{\A^1} (Y)\rightarrow 1\]
	
	Suppose that $\@L$ admits an $n$-th root so that $\@L=\@M^{\otimes n}$ for some $\@M\in \Pic(X)$. Then the map $\@L: \pi_1^{\A^1}(X)\rightarrow \mathbb{G}_m$ factors as
	\[\pi_1^{\A^1}(X)\overset{\@M}{\rightarrow} \mathbb{G}_m\overset{\psi^n}{\rightarrow} \mathbb{G}_m\]
	As $\psi^n$ is not surjective, we see that $\@L$ is not surjective. Hence, by the long exact sequence $\pi_0^{\A^1}(Y)$ cannot be trivial.
	
\end{proof}


\begin{theorem}\label{theorem-mun-gerbes-Pn}
	Let $\P^d\sqrt[n]{\sO(k)}$, $0 < k <n$ be a non trivial $\mu_n$-gerbe over $\P^d$. Then
	\begin{enumerate}
		\item $\P^d\sqrt[n]{\sO(k)}$ is $\A^1$-connected if $k$ and $n$ are coprime.
		\item $\pi_{0}^{\A^1}(\P^d\sqrt[n]{\sO(k)}) \simeq \pi_{0}^{\A^1}(B\mu_l)$, where $l = gcd(k,n) >1$
	\end{enumerate} 
\end{theorem}
\begin{proof}
We have the following cartesian square 
$$\xymatrix{
	B\mu_n \ar[r]^{id} \ar[d] &	B\mu_n \ar[d]\\
	\P^d\sqrt[n]{\sO(k)} \ar[r] \ar[d] &  B\mathbb{G}_m\ar[d] \\
	\P^d\ar[r]^{\sO(k)} &  B\mathbb{G}_m
}$$ where the lower vertical maps are $\A^1$-fibrations. So we have a long exact sequence induced by the $\A^1$-fibration $ B\mu_n \to 	\P^d\sqrt[n]{\sO(k)}  \to \P^{d}$.   $$ 1 \to\mu_n \to \pi_1^{\A^1}(\P^d\sqrt[n]{\sO(k)} ) \to \pi_1^{\A^1}(\P^d)\to \pi_0^{\A^1}(B\mu_n) \to \pi_0^{\A^1}(\P^d\sqrt[n]{\sO(k)})\to 1$$ So we have a surjective map $\pi_0^{\A^1}(B\mu_n) \to \pi_0^{\A^1}(\P^d\sqrt[n]{\sO(k)})$ and by exactness it suffices to understand image of the morphism $\pi_1^{\A^1}(\P^d)\to \pi_0^{\A^1}(B\mu_n)$. We have the following commutative square induced by the cartesian diagram above  $$\xymatrix{
\pi_1^{\A^1}(\P^d)\ar[r] \ar[d] &  \mathbb{G}_m \ar[d]^{\rho_{n}} \\
\pi_0^{\A^1}(B\mu_n) \ar[r]^{id}&  \pi_0^{\A^1}(B\mu_n) 
}$$ where the upper horizontal map has image $\mathbb{G}_m^k
$ as $\pi_1^{\A^1}(\P^d) \to \mathbb{G}_m$ is induced by $\P^d \xrightarrow{\sO(k)} B\mathbb{G}_m$. (See the last paragraph of proof of Proposition \ref{3.6}). $ \rho_{n}: \mathbb{G}_m \to \pi_0^{\A^1}(B\mu_n)$ is the surjective map induced by the fibration $B\mu_n \to B\mathbb{G}_m\to B\mathbb{G}_m$ (Corollary \ref{3.7}, Notation \ref{3.9}). Hence by exactness we conclude $\pi_0^{\A^1}(\P^d\sqrt[n]{\sO(k)}) \simeq \pi_0^{\A^1}(B\mu_n)/\rho_{n}(\mathbb{G}_m^k)$. Finally we observe that $\mathbb{G}_m^k \cap \mathbb{G}_m^n = \mathbb{G}_m^l$, where $l = gcd(k,n)$. Therefore $\rho_{n}(\mathbb{G}_m^k) \simeq (\mathbb{G}_m^k)/(\mathbb{G}_m^l)$, so we have  $\pi_0^{\A^1}(\P^d\sqrt[n]{\sO(k)}) \simeq \pi_0^{\A^1}(B\mu_n)/\rho_{n}(\mathbb{G}_m^k) \simeq \mathbb{G}_m/(\mathbb{G}_m^l) \simeq \pi_{0}^{\A^1}(B\mu_l)$. In case $k$ and $n$ are coprime the $\A^1$-connectedness of the gerbe is also clear.
\end{proof}
\begin{lemma}
	The stack $\P(m,n)$ is $\A^1$-connected.
\end{lemma}
\begin{proof}
	In case $m$ and $n$ are coprime we use Lemma \ref{onestackypoint}, so assume $1< d:=gcd(m,n)$. Then $\P(m,n)$ is a $\mu_d$-gerbe over the $\A^1$-connected stacky curve $\P(m/d,n/d)$. We claim $\P(m,n)$ is the base change of $$\xymatrix{
		\P(m,n) \ar[d]^{p}\ar[r] &  B\mathbb{G}_m \ar[d]^{\psi^d} \\
		\P(m/d, n/d) \ar[r]^{\@L} &  B\mathbb{G}_m 
	}$$ where $\@L$ is the $\mathbb{G}_m$-torsor $\A^2 \setminus 0$ with the action $\lambda(x,y) = (\lambda^{m/d}.x, \lambda^{n/d}.y)$. The morphism $p$ is induced by the map $\A^2\setminus0 \to \A^2\setminus0$ which sends $(x,y) \neq 0$ to $(x^{d}, y^{d})$. Under this map $p^{\ast}(\@L) \simeq (\A^2\setminus0 \times \mathbb{G}_m)/\mathbb{G}_m$ with action $\lambda.(x,y,z) = (\lambda^m.x,\lambda^n.y,\lambda^d.z)$. By Definition \ref{redstr}, $p^{\ast}(\@L) $ is the $d^{th}$ power of the torsor $\A^2\setminus \to \P(m,n)$. Let $\sX$ be the stack which is the fibered product of the above diagram. By the description of fibered products $\sX$ is a functor whose sections on a scheme $Y$ consists of $\{\alpha: Y \to \P(m/d,n/d), \@M \in \Pic(Y), \beta: \@M^{\otimes d} \simeq \alpha^{\ast}\@L\}$ where as a morphism $ Y \to \P(m,n)$ consists of $\{\beta: T \to \A^2\setminus 0  \} $, where $T \to Y$ is $\mathbb{G}_m$-torsor and $\beta$ is an equivariant map to $\A^2\setminus 0$ with action $\lambda.(x,y) = (\lambda^m.x, \lambda^n.y)$. These two data are equivalent as given $Y \to \sX$ we set $\@M$ as $T$ and since $\@M^{\otimes d} \simeq \alpha^{\ast}\@L$ we also get a  map $T \to \A^2\setminus 0 $ satisfying the required equivariance condition. Conversely given $\gamma :Y \to \P(m,n)$ we compose with $p: \P(m,n) \to \P(m/d,n/d)$ to get $\alpha$. Set $\@M$ as the pullback of $\A^2 \setminus 0 \to \P(m,n)$ under $p$. Since $\A^2 \setminus 0 \to \P(m,n)$ is $d^{th}$ power of $p^{\ast}(\@L)$ in $\Pic(\P(m,n))$, we obtain $\@M^{\otimes d} \simeq \alpha^{\ast}\@L$. 
	So by Lemma \ref{3.7}, we have an exact sequence $$ \cdots \to  \pi_1^{\A^1}(\P(m/d,n/d)) \to \pi_0^{\A^1}(B\mu_d) \to \pi_0^{\A^1}(\P(m,n)) \to 1$$ The theorem will follow from the next claim.
	\begin{claim}
		$\pi_0^{\A^1}(B\mu_d) \to \pi_0^{\A^1}(\P(m,n)) $ is the zero map.
	\end{claim}
	By exactness it suffices to prove $\pi_1^{\A^1}(\P(m/d,n/d)) \to \pi_0^{\A^1}(B\mu_d)$ is surjective.  $\pi_1^{\A^1}(\P(m/d,n/d)) \to \pi_0^{\A^1}(B\mu_d)$ factors as $\pi_1^{\A^1}(\P(m/d,n/d)) \xrightarrow{h} \mathbb{G}_m \xrightarrow{\rho_d} \pi_0^{\A^1}(B\mu_d)$, where $h$ is induced by $\@L: \P(m/d,n/d) \to B\mathbb{G}_m$. By the dicussion above $\@L$ is given by $\A^2\setminus0$ which is $\A^1$-connected and so $h$ is surjective. The morphism $\rho_d$ is already known to be surjective so   $\pi_1^{\A^1}(\P(m/d,n/d)) \to \pi_0^{\A^1}(B\mu_d)$ is surjective and we are done. 
\end{proof}
\begin{corollary}\label{theorem-elliptic-pi0}
The moduli stack of elliptic curves over a field of characteristic $\neq 2,3$-denoted $\overline{\@M}_{1,1}$ is $\A^1$-connected.
\end{corollary}
\begin{proof} This follows from the previous lemma and the isomorhism $\overline{\@M}_{1,1} \simeq \P(4,6)$.
\end{proof}
	Following corollary is an immediate consequence of \cite[Proposition 4.3.8]{amorel}. The assumption of a scheme in the proposition in op. cit is superfluous. It holds for any $\A^1$-connected space.
	
\begin{corollary}\label{corollary-brauer-M11}
We have the following isomorphisms for the Brauer group of $\overline{\@M}_{1,1}$ over a field $k$ \begin{enumerate}
		\item  $\mathrm{Br}(\overline{\@M}_{1,1}) \simeq  \mathrm{Br}(k)$  if char($k$) $=0$
		\item $\mathrm{Br}(\overline{\@M}_{1,1})[\frac{1}{p}] \simeq  \mathrm{Br}(k)$ if  char($k$) $ >3$
\end{enumerate}
\end{corollary}
	
	\begin{theorem}\label{4.15}
		The only $\A^1$-connected stacky curves over $\mathbb{C}$ are
		\begin{enumerate}
			\item \label{co1} The projective line $\P^1$.
			\item \label{co2} The stacky curves $\P(m,n)$.
		\end{enumerate}
	\end{theorem}
\begin{proof}
	The only thing left to prove is that any $\A^1$-connected stacky curve is isomorphic to $\P(m,n)$ for some $m,n \geq 1$. By \cite[Proposition 4.1.1]{amorel} any $\A^1$-connected space over an algebraically closed field of characteristic $0$ has the trivial $\acute{e}$tale fundamental group. By \cite[Theorem 1.1]{bn} these are precisely (\ref{co1}) and (\ref{co2}).
\end{proof}

\section{Homotopy purity for stacks}

In this section we explain homotopy purity for algebraic stacks. This essentially follows from descent once an appropriate scheme approximation is available in the Morel--Voevodsky category. For more general formulations of the homotopy purity property, including the stable version, see \cite{cda}.

\begin{theorem}\label{theorem-homotopy-purity}
	Let $\sX$ be a smooth algebraic stack over a scheme $S$ and let $\sZ\subset\sX$ be a smooth closed substack.
	Then we have an equivalence,
	\[\sX/(\sX\setminus\sZ)\simeq Th(N_{\sZ}(\sX))\]
	in $\sH(S)$.
\end{theorem}

\begin{proof}
	Let $X\rightarrow \sX$ be a smooth-Nisnevich covering. Then the associated \v{C}ech nerve is Nisnevich locally equivalent to $\sX$, i.e, $X_{\bullet}\simeq \sX$ along the augmentation map (see \cite{des23} for details).
	
	Let $\sU:= \sX\setminus \sZ$. Restricting $X_{\bullet}$ to $\sU$ and $\sZ$ produces weak equivalences $U_{\bullet}\rightarrow\sU$ and $Z_{\bullet}\rightarrow\sZ$. 
	
	Consider the cofibre sequence,
	\[\sU\rightarrow\sX\rightarrow \sX/\sU\]
	This sequence is equivalent to the sequence,
	\begin{equation}\label{equation-cofiber-simplicial-purity}
	U_{\bullet}\rightarrow X_{\bullet}\rightarrow X_{\bullet}/U_{\bullet}
	\end{equation}
	Note that the above cofibre can be constructed by taking the cofibre levelwise. Further,
	observe that for each $i$, we have $U_i:= X_i\setminus Z_i$ by construction and that the scheme $Z_i$ is a smooth closed subscheme of $X_i$. Thus, by homotopy purity for schemes, $X_i/U_i\simeq Th(N_{Z_i}(X_i))$. Combining this with the cofibre sequence (\ref{equation-cofiber-simplicial-purity}), we have an equivalence
	\begin{equation}\label{equation-Thom-equals-cofibre}
	\hocolim Th(N_{Z_i}(X_i))\simeq X_{\bullet}/U_{\bullet}.
	\end{equation}
	Thus, it remains to show that $\hocolim Th(N_{Z_i}(X_i))\simeq Th(N_\sZ)(\sX)$. To see this, we begin by unwinding the definitions.
	The Thom space of the normal bundle can be described as the cofibre,
	\[N_\sZ(\sX)\setminus 0\rightarrow N_\sZ(\sX)\rightarrow Th(N_\sZ)(\sX)\]
	
	Futher, as $X\rightarrow\sX$ is smooth, we have a cartesian square,
	\begin{center}
		\begin{tikzcd}
			N_Z(X)\arrow[r]\arrow[d] & X\arrow[d]\\
			N_\sZ(\sX)\arrow[r] & \sX
		\end{tikzcd}
	\end{center}
	where $Z$ is the pullback of $\sZ$ to $X$. As smooth-Nisnevich covers are preserved under pullback, taking the \v{C}ech Nerve gives an equivalence $\lbrace N_{Z_{i}}(X_{i})\rbrace_{\bullet}\simeq N_{\sZ}(\sX)$. Taking Thom spaces objectwise for the \v{C}ech nerve $\lbrace N_{Z_{i}}(X_{i})\rbrace_{\bullet}$ produces the simplicial object $\lbrace Th(N_{Z_i}(X_i))\rbrace_{\bullet}$. This object is equivalent to $Th(N_\sZ)(\sX)$ by construction. Combining this with (\ref{equation-Thom-equals-cofibre}), we get that desired equivalence.
\end{proof}

\begin{corollary}
	Let $\sX$ be $\P^1$ with one stacky point of Example \ref{onestackypoint} defined over an algebraically closed field. Let $\sZ= B\mu_n$ be the residual gerbe at the stacky point. Then $\sX\simeq Th(N_\sZ (\sX))$ in $\sH(k)$.
\end{corollary}
\begin{proof}
	This immediately follows from homotopy purity (Theorem \ref{theorem-homotopy-purity}). The inclusion $\sZ$ is a smooth closed substack of $\sX$ with complement $\sX\setminus \sZ \simeq \A^1$ which is contractible. Thus, $\sX/(\sX\setminus \sZ)\simeq \sX/pt\simeq \sX$. Now the cofibre sequence,
	\[\sX\setminus \sZ\rightarrow \sX\rightarrow  \sX/(\sX\setminus\sZ)\]
	yields the equivalence $\sX\simeq Th(N_\sZ (\sX))$ using Theorem \ref{theorem-homotopy-purity}.
\end{proof}

\section{computation of motivic cohomology} In this section we analyse the motive of a smooth tame stacky curve with a smooth coarse space. We have the following result 

\begin{proposition} Let $\sX$ be a tame smooth stacky curve with a smooth coarse space $X$ over a
field $k$. Let $\sZ$ denote the stacky locus of $\sX$ and  $\sZ\rightarrow Z$ be the corresponding coarse moduli space. Then we have a cartesian diagram of motives
\end{proposition}

\begin{center}
	\begin{tikzcd}
	M(\sX)\arrow[r]\arrow[d] & M(\sZ)(1)[2]\arrow[d]\\
	M(X) \arrow[r] & M(Z)(1)[2]
	\end{tikzcd}
\end{center}

\begin{proof} 
Note that as $X$ and $\sX$ have dimension one, $Z$ is a finite set of points and $\sZ$ is a
finite set of gerbes over the points of $Z$. Moreover, the coarse space morphism $\sX\rightarrow X$ is an isomorphism outside the stacky locus. Using the Gysin triangle for motives, we have a
morphism of triangles
\begin{center}
	\begin{tikzcd}
	\arrow[r]& M(\sX)\arrow[r]\arrow[d] & M(\sZ)(1)[2]\arrow[r]\arrow[d] & M(U)[1] \arrow[r]\arrow[d,"="] & \ldots\\
	\arrow[r]& M(X)\arrow[r] & M(Z)(1)[2] \arrow[r] & M(U)[1]\arrow[r] &\ldots
	\end{tikzcd}
\end{center}
where the right vertical map is an isomorphism. A simple diagram chase now yield the
result.
\end{proof}



\begin{remark} Let $\sX$ be a tame smooth stacky curve with smooth coarse space $X$. Let $\sZ$ be
the finite set of stacky points of $\sX$. Then $\sZ = \sqcup_i B\Z/n_i\Z$, and the map $p : \sZ\rightarrow Z$ is the degree
 $1/n_i$ map $p_i : B\Z/n_i\Z \rightarrow \Spec(k)$ on each component. Then applying $\Hom(-,\Z(1)[2])$ gives
us a diagram,
\begin{center}
	\begin{tikzcd}
	\oplus_i \Hom(M(\Spec k)(1)[2],\Z(1)[2])\arrow[r]\arrow[d] & \oplus_i \Hom(M(B\Z/n_i\Z)(1)[2],\Z(1)[2])\arrow[d]\\
	\Hom(M(X),\Z(1)[2]) \arrow[r] & \Hom(M(\sX),\Z(1)[2])
	\end{tikzcd}
\end{center}
Note that by Voevodsky’s cancellation theorem, \[\Hom(M(Y)(1)[2],\Z(1)[2]) \simeq
\Hom(M(Y),\Z) \simeq \Z^{\oplus r}\]
where $r$ is the number of connected components. And, by \cite[Corollary 4.2]{mwv}, \[\Hom(M(Y),\Z(1)[2])\simeq \Pic(Y).\]
Further, the vertical maps are given by sending the generators of $\Z^{\oplus r}$ to the line bundles associated to the connected components. Under these maps the image of the morphism $\oplus_i \Hom(M(Z),\Z)\rightarrow \Pic(\sX)$ is given by a line bundle for which the stabilisers act trivially on the fibres. This implies that the top horizontal arrow is a multiplication $n_i$ on the connected component of $B\Z/n_i\Z$.

\begin{center}
	\begin{tikzcd}
	\Z^{\oplus r}\arrow[r,"{(n_1,\ldots, n_r)}"]\arrow[d] & \Z^{\oplus r}\arrow[d]\\
	\Pic(X) \arrow[r] & \Pic(\sX)
	\end{tikzcd}
\end{center}
This gives an alternative proof of \cite[Theorem 1.1]{lopez}.

\end{remark}

\begin{corollary}\label{5.3}
	Let $\sX$ be a stacky curve with $r$ stacky points having residual gerbes $B\mu_{n_1} \cdots B\mu_{n_r}$, $n_i$ pairwise coprime and $\P^1$ as the coarse space. Then the resulting short exact sequence of Picard groups is $$ 1 \to \Z \xrightarrow{i} \Z \xrightarrow{j} \prod_{i=1}^r \Z/n_i\Z\to 1$$ where $i(1) =m= n_1.n_2\dots n_r$ and $j$ is the product of canonical maps $\Z \to \Z/n_iZ$.
\end{corollary}

\begin{proof}
	We use the concrete description of pushouts in the category of abelian groups, namely given a pushout square of abelian groups
	\begin{center}
		\begin{tikzcd}
			A \arrow[r,"f"]\arrow[d,"g"] &  B\arrow[d]\\
			C \arrow[r] & D
		\end{tikzcd}
	\end{center} then $D \simeq (B \oplus C)/f-g$. Applying this to the square we have $\Pic(\sX) \simeq \Pic(\P^1) \oplus \Z^{\oplus r}/f-g$ where $f:  \Z^{\oplus r} \to \Pic(\P^1) \simeq \Z$ is given by $f(x_1, \cdots, x_r) = x_1 + \cdots +x_r$ and $g:  \Z^{\oplus r} \to \Z^{\oplus r} $ is given by $g(x_1, \cdots, x_r) = (n_1x_1, \cdots, n_rx_r)$. Define $\Pic(\sX) \to\Z$ by sending $(a_1, b_1, \cdots b_r) \mapsto m.a_1 + m/n_1.b_1 + \cdots + m/n_r.b_r$ which is easily seen to be an isomorphism and rest of the result follows.
\end{proof}

\end{document}